\documentclass[12pt]{article}
\usepackage[affil-it]{authblk}
\usepackage{amsmath}
\usepackage{amssymb}
\usepackage{amsfonts}
\usepackage{amsthm}
\usepackage{pgf}
\usepackage{tikz}
\usepackage[bookmarks=true,pdfborder={0 0 0}]{hyperref}

\newtheorem{theorem}{Theorem}[section]
\newtheorem{lemma}[theorem]{Lemma}
\newtheorem{corollary}[theorem]{Corollary}

\newtheorem{obs}[theorem]{Observation}

\theoremstyle{definition}
\newtheorem{definition}[theorem]{Definition}
\newtheorem{example}[theorem]{Example}
\newtheorem{notation}[theorem]{Notation}

\theoremstyle{remark}
\newtheorem{remark}[theorem]{Remark}

\numberwithin{equation}{section}

\DeclareMathOperator{\match}{match}

\DeclareMathOperator{\tr}{tr}
\DeclareMathOperator{\jac}{Jac}
\DeclareMathOperator{\atA}{\;{\rule[-3.6mm]{.1mm}{8mm}}_{A}}
\renewcommand{\i}{\mathrm{i}}

\begin{document}
\title{Construction of real skew-symmetric matrices from interlaced spectral data, and graph}
\author{Keivan Hassani Monfared\hspace{15pt} Sudipta Mallik }
\affil{\small Department of Mathematics, Western Illinois University, 476 Morgan Hall, 1 University Circle, Macomb, IL 61455, USA\\ Department of Mathematics \& Statistics, Northern Arizona University, 805 S. Osborne Dr. PO Box: 5717, Flagstaff, AZ 86011, USA\\ E-mail addresses: k-hassanimonfared@wiu.edu and Sudipta.Mallik@nau.edu}
\date{}
\maketitle   

\renewcommand{\thefootnote}{\fnsymbol{footnote}} 
\footnotetext{\emph{2010 Mathematics Subject Classification. 05C50,65F18}}    
\renewcommand{\thefootnote}{\arabic{footnote}} 

\begin{abstract}
A 1989 result of Duarte asserts that for a given tree $T$ on $n$ vertices, a fixed vertex $i$, and two sets of distinct real numbers $L,M$ of sizes $n$ and $n-1$, respectively, such that $M$ strictly interlaces $L$, there is a real symmetric matrix $A$ such that graph of $A$ is $T$, eigenvalues of $A$ are given by $L$, and eigenvalues of $A(i)$ are given by $M$. In 2013, a similar result for connected graphs was published by Hassani Monfared and Shader, using the Jacobian method. Analogues of these results are presented here for real skew-symmetric matrices whose graphs belong to a certain family of trees, and all of their supergraphs.
\end{abstract}

\section{Introduction}

Inverse eigenvalue problems (IEP's) have long been studied because of many applications that they have in various areas of science and engineering \cite{boleyGolub87, chu98}. That is, to find a matrix in a certain family of matrices with prescribed eigenvalues, eigenvectors, or both. In particular, structured inverse eigenvalue problems (SIEP's) have received a lot of attention \cite{chuGolub02}. For example, one might be interested in finding matrices which have prescribed eigenvalues where the solution matrix has a certain zero-nonzero pattern. In this paper we study an SIEP which asks about the existence of a real skew-symmetric matrix with a specific zero-nonzero pattern where the eigenvalues of the matrix and the eigenvalues of a principal submatrix of it are prescribed and are distinct. We shall give a precise formulation of the problem (which we call the $\lambda-\mu$ skew-symmetric SIEP), and a solution when the structure of the matrix is defined by a family of trees and their supergraphs.

Cauchy interlacing inequalities \cite{cauchy} assert that the eigenvalues of a real symmetric matrix and those of a principal submatrix of it satisfy certain inequalities. Namely, if $A$ is an $n\times n$ real symmetric matrix with eigenvalues $\lambda_1 \leq \lambda_2  \leq \cdots \leq \lambda_n$, and $B$ is an $(n-1)\times (n-1)$ principal submatrix of $A$ with eigenvalues $\mu_1 \leq \mu_2 \leq \cdots \leq \mu_{n-1}$. Then
\begin{equation}\label{cauchysymm}
\lambda_1\leq \mu_1\leq \lambda_2\leq\cdots \leq \mu_{n-1}\leq \lambda_n.
\end{equation} 

Note that Cauchy interlacing inequalities hold for any Hermitian matrix $A$ in general. The {\it spectrum} of a square matrix $A$, denoted by $\sigma(A)$, is the set of eigenvalues of $A$.  For the preceding $A$ and $B$, we say that $\sigma(B)$ {\it interlaces}  $\sigma(A)$. If the inequalities in (\ref{cauchysymm}) are all strict, we say $\sigma(B)$ \textit{strictly interlaces} $\sigma(A)$. 

Here we introduce similar inequalities to Cauchy interlacing inequalities for skew-symmetric matrices. Since all the eigenvalues of any skew-symmetric matrix are purely imaginary numbers, we define an ordering on the imaginary axis of the complex plane. Let $\mathrm{i}$ denote the complex number $\sqrt{-1}$ and $\i\mathbb R=\{\i a: a\in \mathbb R\}$. For $a,b \in \i \mathbb{R}$ we say $a \leq b$ whenever $-\i a \leq -\i b$, and the equality holds if and only if $a=b$. Let $\mathcal S=\{a_1,\ldots, a_n\}$  be a subset of $\i\mathbb R$. $\mathcal S$ is said to be \textit{presented in increasing order} if $ a_1 \leq  a_2 \leq \cdots \leq  a_n$. Throughout this article we always present spectra of skew-symmetric matrices in increasing order. If $\mathcal S=\{a_1,\ldots,a_n\}$ is presented in increasing order, then $a_1$ is said to be the \textit{smallest element} of $\mathcal S$, $a_2$ is said to be the \textit{second smallest element} of $\mathcal S$, and so on.

Let $\mathcal A=\{\lambda_1,\ldots,\lambda_n\}$ and $\mathcal B=\{\mu_1,\ldots,\mu_{n-1}\}$ be subsets of $\i \mathbb{R}$, presented in increasing order. $\mathcal B$ is said to \textit{interlace} $\mathcal A$ if  $\lambda_1 \leq \mu_1 \leq \lambda_2 \leq \cdots \leq \mu_{n-1} \leq \lambda_n$. Similarly $\mathcal B$ is said to \textit{strictly interlace} $\mathcal A$ if  $\lambda_1 < \mu_1 < \lambda_2< \cdots < \mu_{n-1}< \lambda_n$. Now Cauchy interlacing inequalities for skew-symmetric matrices can be stated as follows.

\begin{theorem}[Cauchy interlacing inequalities for skew-symmetric matrices] Let $A$ be an $n\times n$ real skew-symmetric matrix and $B$ be an $(n-1)\times (n-1)$ principal submatrix of $A$. Then $\sigma(B)$ interlaces $\sigma(A)$.
\end{theorem}
\begin{proof}
Let $A$ be an $n\times n$ real skew-symmetric matrix and $B$ be an $(n-1)\times (n-1)$ principal submatrix of $A$. Let $\sigma (A)=\{\lambda_1,\ldots,\lambda_n\}$ and $\sigma (B)=\{\mu_1,\ldots,\mu_{n-1}\}$.

Since $A$ is a real skew-symmetric matrix, $-\i A$ is an $n\times n$ Hermitian matrix and $-\i B$ is an $(n-1)\times (n-1)$ Hermitian matrix, principal submatrix of $-\i A$. Also $\sigma (-\i A)=\{-\i \lambda_1,\ldots, -\i \lambda_n \}$ and $\sigma (-\i B)=\{-\i \mu_1, \ldots, -\i \mu_{n-1}\}$. Now by \ref{cauchysymm} for  $-\i A$ and $-\i B$, we have $-\i \lambda_1 \leq -\i \mu_1 \leq -\i \lambda_2 \leq \cdots \leq -\i \mu_{n-1}\leq -\i \lambda_n$. Thus $\sigma(B)$ interlaces $\sigma(A)$.
\end{proof}

Let $A=[a_{i,j}]$ be an $n\times n$  real symmetric or skew-symmetric matrix. We say $A$ is of order $n$, and denote it by $|A|=n$.  The graph of $A$, denoted by $G(A)$, has the vertex set $\{1,2,\ldots,n\}$ and the edge set  $\{\{i,j\}: a_{i,j}\neq 0, 1\leq i<j\leq n\}$. $S(G)$ denotes the set of all real symmetric matrices whose graph is $G$. Similarly $S^-(G)$ denotes the set of all real skew-symmetric matrices whose graph is $G$.

For a vertex $v$ of $G$, the set of all vertices of $G$ that are adjacent to $v$ is denoted by $N(v)$. For a vertex $w$ of a tree $T$, $T(w)$ denotes the forest obtained from $T$ by deleting the vertex $w$. If $v$ is a neighbor of $w$ in $T$, then $T_v(w)$ denotes the connected component of $T(w)$ having $v$ as a vertex. Note that $T_v(w)$ is a tree. For $A$ in $S(T)$ or $S^-(T)$, $A(w)$ denotes the principal submatrix of $A$ corresponding to $T(w)$ and $A_v(w)$ denotes the principal submatrix of $A$ corresponding to $T_v(w)$.  The graph obtained from $T_v(w)$ by deleting vertex $v$ is denoted by $T_{v'}(w)$ and $A_{v'}(w)$ denotes the principal submatrix of $A_v(w)$ corresponding to $T_{v'}(w)$. Also, for any matrix $A$, $C_A(x)$ denotes the characteristic polynomial of $A$.

\begin{figure}[h]
\begin{center}
\begin{tikzpicture}[node distance=1.5cm,colorstyle/.style={circle, draw=black!100,fill=black!100, thick, inner sep=0pt, minimum size=2 mm}]
 \node[colorstyle,label=above:$6$] (1) {};
 \node[colorstyle,label=above:$2$] (2) [right of=1] {};
 \node[colorstyle,label=above right:$4$] (4) [below right of=2] {};
 \node[colorstyle,label=above:$5$] (5) [left of=1] {};
 \node[colorstyle,label=left:$3$] (3) [below of=1] {};
 \node[colorstyle,label=left:$1$] (6) [below of=3] {};
 \draw[-,thick] (1) -- (2);
 \draw[-,thick] (2) -- (4);
 \draw[-,thick] (1) -- (3);
 \draw[-,thick] (1) -- (5);
 \draw[-,thick] (3) -- (6);
\node[] () at (0,-3.9) {$T$};

 \node[](6) at (7,0){};
 \node[colorstyle,label=above right:$2$] (7) [right of=6] {};
 \node[colorstyle,label=above right:$4$] (8) [below right of=7] {};
 \node[colorstyle,label=above:$5$] (9) [left of=6] {};
 \node[colorstyle,label=left:$3$] (10) [below of=6] {};
 \node[colorstyle,label=left:$1$] (11) [below of=10] {};
 \draw[-,thick] (7) -- (8);
 \draw[-,thick] (10) -- (11);
 \node[] () at (8, -3.9){$T(6)$};
 \draw[dashed, very thick] (9,-.5) circle (1.1) {};
 \node[] () at (9,1) {$T_2(6)$};
 \draw[dashed, very thick] (6.7,-2.2) circle (1.1) {};
 \node[] () at (8.3,-2.3) {$T_3(6)$};
 \draw[dashed, very thick] (5.8,.3) circle (.8) {};
 \node[] () at (5,-.7) {$T_5(6)$};
\end{tikzpicture}
\caption{Tree $T$, and its subtrees $T_j(6)$, $j=2,3,5$.}\label{duartetree}
\end{center}
\end{figure}
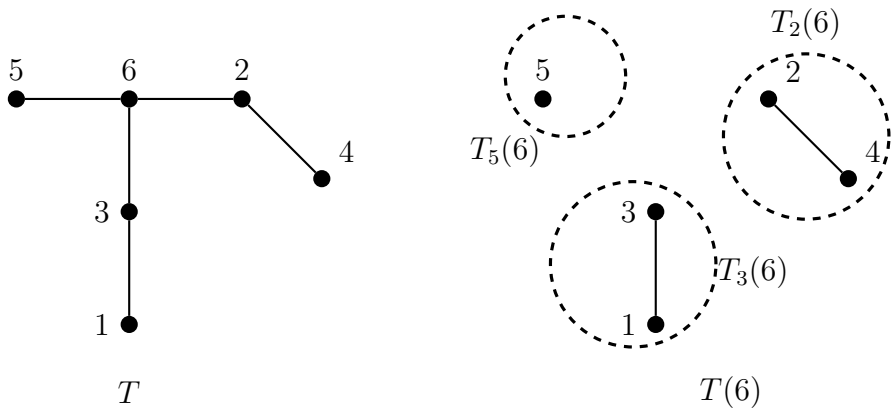

The following result is obtained by Duarte \cite{duarte}.
\begin{theorem}\label{duarte}
Let $T$ be a tree on $n$ vertices $1, 2,\ldots,n$ with $n\geq 2$. Let $\lambda_1,\lambda_2, \ldots, \lambda_n, \mu_1,\ldots,\mu_{n-1}$ be $2n-1$ real numbers such that $\lambda_1< \mu_1< \lambda_2<\cdots <\mu_{n-1}< \lambda_n$.  Then there exists a symmetric matrix $A$ in $S(T)$ with eigenvalues $\lambda_1,\lambda_2, \ldots, \lambda_n$ such that the eigenvalues of $A(1)$ are $\mu_1,\mu_2, \ldots, \mu_{n-1}$.
\end{theorem}
Later, Hassani Monfared and Shader \cite{ms} extended Theorem \ref{duarte} to connected graphs. In this article we prove analogous results for real skew-symmetric matrices, with some combinatorial (sufficient) restrictions on the graph of the matrix. The structure of the paper is as follows.

In Section \ref{sectionduarteNEB} trees with nearly even branching at a vertex $v$ are defined, and it is shown that having the Duarte property with respect to a vertex in a matrix $A$ implies the nearly even branching property at that vertex for $G(A)$.

In Section \ref{sectiontree} we study the characteristic polynomials of a skew-symmetric matrix whose graph is a tree, and its relation to the characteristic polynomial of one of its principal submatrices. It is shown that if a tree has a nearly even branching property at a vertex $v$, then for any set of distinct eigenvalues that satisfy some necessary conditions, the $\lambda-\mu$ skew-symmetric SIEP has a solution. Furthermore, the solution has the Duarte property with respect to vertex $v$.

In Section \ref{sectionjacobian} we define and study a function that takes a matrix $A$ and maps it to its characteristic polynomial and the characteristic polynomial of a principal submatrix, $A(v)$. It is shown that this map has a nonsingular Jacobian, when it is evaluated at a point corresponding to a matrix with the Duarte property with respect to $v$.

Finally, in Section \ref{sectionconnected} we extend the result for trees with nearly even branching property at a vertex to their supergraphs with the aid of the Implicit Function Theorem.

\section{The Duarte property and trees with the nearly even branching property}\label{sectionduarteNEB}
In this section we define a special property, called the Duarte-property \cite{ms}, of a square matrix whose graph is a tree. Then we define a certain family of trees and discuss its properties.

\begin{definition}\label{DuarteProperty}
Let $A$ be an $n\times n$ matrix whose graph is a tree. If $G(A)$ has just one vertex, then $A$ has the Duarte-property with respect to $w$. 
 If $G(A)$ has more than one vertex, then $A$ has the Duarte-property with respect to $w$ provided the eigenvalues of $A(w)$ strictly interlace those of $A$ and for each neighbor $v$ of $w$, $A_v(w)$ has the Duarte-property with respect to the vertex $v$.
\end{definition}

\begin{example}
Consider the matrix $A$ below whose graph is $T$.

\begin{center}
\begin{tikzpicture}[node distance = 1cm, scale=.8,colorstyle/.style={circle, draw=black!100,fill=black!100, thick, inner sep=0pt, minimum size=1.2 mm}]

	\node[] () at (-7,-2) {$A = \left[\begin{array}{rrrrr}
0 & 8 & 0 & 0 & 0 \\
-8 & 0 & 4 & 0 & 1 \\
0 & -4 & 0 & 2 & 0 \\
0 & 0 & -2 & 0 & 0 \\
0 & -1 & 0 & 0 & 0
\end{array}\right],$};
	
	\node[] () at (-2,-2) {$T:$};
	\node[colorstyle,label=above:$1$] (1) {};
	\node[colorstyle,label=right:$2$] (2) [below of=1] {};
	\node[colorstyle,label=right:$5$] (5) [below right of=2] {};
	\node[colorstyle,label=right:$3$] (3) [below left of=2] {};
	\node[colorstyle,label=right:$4$] (4) [below of=3] {};
	
	\draw[-,thick] (1) -- (2) -- (3) -- (4);
	\draw[-,thick] (2) -- (5);
\end{tikzpicture}
\end{center}

Then 

\begin{center}
\begin{tikzpicture}[node distance = 1cm, scale=.8,colorstyle/.style={circle, draw=black!100,fill=black!100, thick, inner sep=0pt, minimum size=1.2 mm}]
	\node[] () at (-7,-1) {$A(1) = \left[\begin{array}{rrrr}
 0 & 4 & 0 & 1 \\
 -4 & 0 & 2 & 0 \\
 0 & -2 & 0 & 0 \\
 -1 & 0 & 0 & 0
\end{array}\right],$};
	
	\node[] () at (-2,-1) {$T(1):$};

	\node[colorstyle,label=right:$2$] (2) [] {};
	\node[colorstyle,label=right:$5$] (5) [below right of=2] {};
	\node[colorstyle,label=right:$3$] (3) [below left of=2] {};
	\node[colorstyle,label=right:$4$] (4) [below of=3] {};
	
	\draw[-,thick] (2) -- (3) -- (4);
	\draw[-,thick] (2) -- (5);
\end{tikzpicture}
\end{center}

The eigenvalues of $A$ are approximately $0$, $\pm 9.05 \, \i$, and $\pm 1.78 \, \i$. The eigenvalues of $A(1)$ are approximately $\pm 4.56 \, \i$ and $\pm 0.44 \, \i$, which strictly interlace those of $A$. The eigenvalue of $(A_2(1))_5(2)= \begin{bmatrix}0\end{bmatrix} $ is $0$, and the eigenvalues of $(A_2(1))_3(2) = \left[ \begin{array}{rr}0 & 2 \\ -2 & 0 \end{array}\right] $ are $\pm 2 \, \i$, which both strictly interlace those of $A_2(1)$. And finally, the eigenvalue of $((A_2(1))_3(2))_4(3)=\begin{bmatrix}0\end{bmatrix} $ is $0$ which strictly interlace those of $(A_2(1))_3(2)$. Thus, $A$ has the Duarte property with respect to vertex $1$.
\end{example}

\begin{definition}\label{NEBtree}
Let $T$ be a tree on $n$ vertices, and $w$ be a vertex of $T$. $T$ is defined to have \textit{nearly even branching property at $w$} (in short, $T$ is {\it NEB at $w$}) as follows. If $n=1$, $T$ is NEB at $w$.  If $n\geq 2$, $T$ is NEB at $w$ if the following conditions are satisfied:
\begin{itemize}
\item[(i)] $T(w)$ has exactly one odd component if $n$ is even, and $T(w)$ has no odd component if $n$ is odd; and
\item[(ii)] for each neighbor $v$ of $w$ in $T$, $T_v(w)$ is NEB at $v$. 
\end{itemize} 
\end{definition}

\begin{figure}[h]
\begin{center}
\begin{tikzpicture}[scale=.8,colorstyle/.style={circle, draw=black!100,fill=black!100, thick, inner sep=0pt, minimum size=1.2 mm}]
\node (1) at (0,0)[colorstyle,label=right:$v$]{};
\node (2) at (0,-1)[colorstyle]{};
\draw [thick] (1)--(2);
\node () at (0,1) {$P$};

\node (1) at (2,0)[colorstyle,label=right:$v$]{};
\node (2) at (2,-1)[colorstyle]{};
\node (3) at (2,-2)[colorstyle]{};
\draw [thick] (1)--(2)--(3);
\node () at (2,1) {$Q$};

\begin{scope}[xshift=-3.5cm]
\node (1) at (9,0)[colorstyle,label=right:$v$]{};
\node (2) at (8.3,-1)[colorstyle]{};
\node (3) at (9,-1)[colorstyle]{};
\node (4) at (9.7,-1)[colorstyle]{};
\node (5) at (8.3,-2)[colorstyle]{};
\node (6) at (9,-2)[colorstyle]{};
\draw [thick] (5) -- (2)--(1)--(3) -- (6);
\draw [thick] (4)--(1);
\node () at (9,1) {$T$};
\end{scope}

\begin{scope}
\node (1) at (9,0)[colorstyle,label=right:$v$]{};
\node (2) at (8.3,-1)[colorstyle]{};
\node (3) at (9,-1)[colorstyle]{};
\node (4) at (9.7,-1)[colorstyle]{};
\node (5) at (8.3,-2)[colorstyle]{};
\node (6) at (9,-2)[colorstyle]{};
\node (7) at (9.7,-2)[colorstyle]{};
\node (8) at (8.3,-3)[colorstyle]{};
\draw [thick] (8) -- (5) -- (2)--(1)--(3) -- (6);
\draw [thick] (7) -- (4)--(1);
\node () at (9,1) {$S$};
\end{scope}

\end{tikzpicture} 
\caption{Examples of NEB trees.}\label{NEBex}
\end{center}
\end{figure}
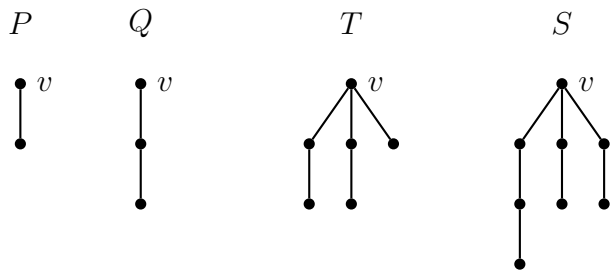

\begin{figure}
\begin{center}
\begin{tikzpicture}[scale=.8,colorstyle/.style={circle, draw=black!100,fill=black!100, thick, inner sep=0pt, minimum size=1.2 mm}]
\node (1) at (4,0)[colorstyle,label=right:$v$]{};
\node (2) at (3.5,-1)[colorstyle]{};
\node (3) at (4.5,-1)[colorstyle]{};
\draw [thick] (2)--(1)--(3);
\node () at (4,1) {$K$};

\node (1) at (6,0)[colorstyle,label=right:$v$]{};
\node (2) at (5.3,-1)[colorstyle]{};
\node (3) at (6,-1)[colorstyle]{};
\node (4) at (6.7,-1)[colorstyle]{};
\draw [thick] (2)--(1)--(3);
\draw [thick] (4)--(1);
\node () at (6,1) {$L$};

\begin{scope}[xshift=5cm,yshift=-1cm]
\node (0) at (4,1)[colorstyle,label=right:$v$]{};
\node (1) at (4,0)[colorstyle,label=right:$w$]{};
\node (2) at (3.5,-1)[colorstyle]{};
\node (3) at (4.5,-1)[colorstyle]{};
\draw [thick] (2)--(1)--(3);
\draw [thick] (0)--(1);
\node () at (4,2) {$F$};
\end{scope}

\node (1) at (12,0)[colorstyle,label=right:$v$]{};
\node (2) at (11.3,-1)[colorstyle]{};
\node (3) at (12,-1)[colorstyle]{};
\node (4) at (12.7,-1)[colorstyle]{};
\node (5) at (11.3,-2)[colorstyle]{};
\node (7) at (12.7,-2)[colorstyle]{};
\node (8) at (11.3,-3)[colorstyle]{};
\draw [thick] (8) -- (5) -- (2)--(1)--(3);
\draw [thick] (7) -- (4)--(1);
\node () at (12,1) {$G$};

\node (1) at (15,0)[colorstyle,label=right:$v$]{};
\node (2) at (14.3,-1)[colorstyle]{};
\node (3) at (15,-1)[colorstyle,label=right:$w$]{};
\node (4) at (16.3,-1)[colorstyle]{};
\node (5) at (14,-2)[colorstyle]{};
\node (6) at (15,-2)[colorstyle]{};
\node (7) at (16,-2)[colorstyle]{};
\node (8) at (14,-3)[colorstyle]{};
\node (9) at (15.4,-3)[colorstyle]{};
\node (10) at (14.6,-3)[colorstyle]{};
\draw [thick] (8) -- (5) -- (2)--(1)--(3) -- (6);
\draw [thick] (7) -- (4)--(1);
\draw [thick] (9) -- (6) --(10);
\node () at (15,1) {$H$};
\end{tikzpicture}
\caption{Nonexamples of NEB trees.}\label{NEBnonex}
\end{center}
\end{figure}
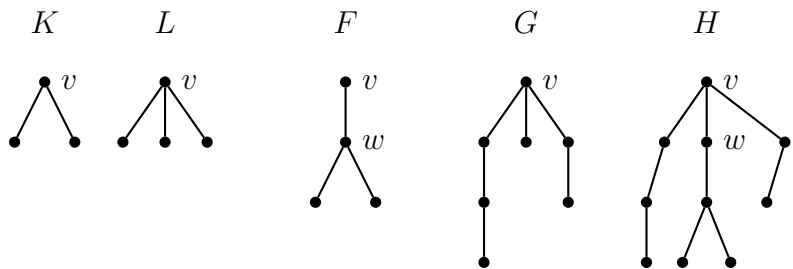
\begin{example}
Every path is NEB with respect to a pendent vertex. Every star with at least $4$ vertices is not NEB with respect to any vertex. In Figure \ref{NEBex}, $P$ is NEB at $v$ since $P(v)$ has only one vertex. Also, $Q$ is NEB at $v$ since $Q(v)$ has only one connected component which is a copy of $P$ and it is shown that $P$ is NEB at its top vertex. Furthermore, $T$ is NEB at $v$ since $T(v)$ has only one odd connected component, and all of its components have either one vertex or are copies of $P$ which are NEB at their top vertices. Similarly, $S$ is NEB at $v$ since $S(v)$ has only one odd connected component (a copy of $Q$) and all of its connected components (two copies of $P$ and one copy of $Q$) are NEB at their top vertices. 

In Figure \ref{NEBnonex}, $K$ is not NEB at $v$, since $K(v)$ has 2 odd connected component (2 isolated vertices). Also, $L$ is not NEB at $v$, since $L(v)$ has 3 odd components (3 isolated vertices). Furthermore, while $F(v)$ has exactly one odd component, $F$ is not NEB at $v$, since $F_w(v)$ is a copy of $K$ which is not NEB at $w$. Moreover, $G$ is not NEB at $v$, since $G(v)$ has 2 odd components (an isolated vertex and a copy of $Q$). Eventually, while $H(v)$ has exactly one odd component, $H$ is not NEB tree at $v$, since $H_w(v)$ is a copy of $F$ which is not NEB at $w$.
\end{example}

The following theorem shows that having the Duarte property implies the nearly even branching property.
\begin{theorem}\label{Duarte2NEB}
If  a tree $T$ of order $n\geq 3$ is not NEB at a vertex $v$, then no $A\in S^-(T)$ has the Duarte property with respect to the vertex $v$.
\end{theorem}
\begin{proof}
Let $T$ be a tree of order $n\geq 3$ which is not NEB at a vertex $v$. Let $N(v)=\{v_1,\ldots,v_k\}$. Let $A$ be in $S^-(T)$ with $\sigma(A)=\{\lambda_1,\ldots,\lambda_n\}$ and $\sigma(A(v))=\{\mu_1,\ldots,\mu_{n-1}\}$. We will induct on the number of vertices. For $n=3$, the only tree on $3$ vertices which is not NEB with respect to a vertex is the path on $3$ vertices, and it is not NEB with respect to the middle vertex. Let $T$ be $K_{1,2}$ in Figure \ref{NEBex}

Then $0\in \sigma(A)$, and also $0\in \sigma(A(v))$. That is, $\sigma(A(v))$ does not strictly interlace $\sigma(A)$. So, $A$ does not have the Duarte property with respect to vertex $v$.

Induction hypothesis: Assume that for any tree $T$ on at most $n-1$ vertices which is not $NEB$ with respect to a vertex $v$, any $A\in S^-(T)$ does not have the Duarte property with respect to $v$.

Now, let $T$ be a tree on $n$ vertices which is not $NEB$ with respect to a vertex $v$, and let $N(v)=\{v_1,\ldots,v_k\}$ be the set of all neighbors of $v$ in $T$. If $\sigma(A(v))$ does not strictly interlace $\sigma(A)$ then we are done. Otherwise there are two cases:

Case 1. $n$ is even.\\
Since $n$ is even and $T$ is not NEB at a vertex $v$, one of the followings is true.
\begin{enumerate}
\item[\rm (a)] $T(v)$ has at least two odd components.
\item[\rm (b)] $T_j(v)$ is not NEB at a vertex  $j$, for some $j\in N(v)$.
\end{enumerate}
First note that since $n$ is even, $\sigma(A)$ does not contain $0$, and also $T(v)$ contains at least one odd component. If (a) holds, then $T_{r}(v)$ and $T_{s}(v)$ are distinct odd components for some distinct $r$ and  $s$ in  $N(v)$. Since $|A_{r}(v)|$ and $|A_{s}(v)|$ are odd, each of $\sigma(A_{r}(v))$ and $\sigma(A_{s}(v))$ contains $0$. That implies multiplicity of $0$ in $\sigma(A(v))$ is at least two, hence by Cauchy interlacing inequalities $0\in \sigma(A)$. Thus, $\sigma(A(v))$  does not strictly interlace $\sigma(A)$. If (b) holds, then $A_j(v)$ does not have the Duarte property with respect to $j$ by the induction hypothesis. Hence $A$ does not have the Duarte property with respect to $v$ by definition. \\

Case 2. $n$ is odd.\\
Since $n$ is odd and $T$ is not a NEB at $v$, one of the followings is true.
\begin{enumerate}
\item[\rm (a)] $T(v)$ has at least one odd component.
\item[\rm (b)] $T_j(v)$ is not NEB at $j$, for some $j\in N(v)$.
\end{enumerate}
First note that since $n$ is odd, $\sigma(A)$ contains $0$. If (a) holds, then $T_{r}(v)$ is an odd component for some $r\in N(v)$. Since $|A_{r}(v)|$ is odd, $\sigma(A_{r}(v))$ contains $0$. Thus $\sigma(A(v))$ does not strictly interlace $\sigma(A)$. Finally if (b) holds, then $A_j(v)$ does not have the Duarte property with respect to $j$ by induction hypothesis. Hence $A$ does not have the Duarte property with respect to $v$ by definition.
\end{proof}

In the following lemma we show that having the Duarte property for a matrix $A$ is very special. That is, the only skew-symmetric matrix which almost commutes with $A$ and has a zero entry whenever $A$ has zero entries, is the zero matrix. The proof is a very similar to that of Lemma 2.2 of \cite{ms}
\begin{lemma}\label{cent}
For matrices $A$ and $B$, let $[A,B]$ denote the commutator of $A$ and $B$, that is, $[A,B]=AB-BA$. Assume that $A$ is a skew-symmetric matrix of order $n$ and $G(A)$ is a tree $T$. Furthermore, assume that $A$ has the Duarte-property with respect to the vertex $w$. Let $X$ be a skew-symmetric matrix such that \begin{itemize}
\item[\rm (a)] $A\circ X=O$, and
\item[\rm (b)] $[A,X](w)=O$.
\end{itemize}
Then $X=O$.
\end{lemma}
\begin{proof}
The proof is by induction on $n$. Without loss of generality we can take $w=1$, and $N(1)=\{ 2,3,\ldots, k+1\}$. For $n\leq 2$, (a)  and the fact that $X$ is skew-symmetric imply that $X=O$.

Induction hypothesis: Assume that for any $m\times m$ matrix $A$, $m \leq l$, for $l < n$, the only skew-symmetric matrix that satisfies conditions \rm{(a)} and  \rm{(b)} is the zero matrix.

Assume $n\geq 3$ and proceed by induction. For $i=1,2,\ldots,k$, let $A_i$ denote the matrix $A_{i+1}(1)$. Then the matrices $A$ and $X$, up to a permutation of rows and columns, have the form 
$$A=\left[ \begin{array}{c|cccc}
0 & b_1^T & b_2^T & \cdots & b_k^T\\\hline
-b_1 & A_1 & O & \cdots & O\\
-b_2 & O & A_2 & \cdots & O\\
\vdots & \vdots & \vdots & \ddots & \vdots\\
-b_k & O & O & \cdots & A_k
\end{array} \right],
\qquad 
 X=\left[ \begin{array}{c|cccc}
0 & u_1^T & u_2^T & \cdots & u_k^T\\\hline
-u_1 & X_{11} & X_{12} & \cdots & X_{1k}\\
-u_2 & X_{21} & X_{22} & \cdots & X_{2k}\\
\vdots & \vdots & \vdots & \ddots & \vdots\\
-u_k & X_{k1} & X_{k1} & \cdots & X_{kk}
\end{array} \right],$$
so that each column vector $b_i$ has exactly one nonzero entry since it is a tree. Without loss of generality we take this nonzero entry of each $b_i$ to be in its first position. Thus the $A_i$'s correspond to the $T_v(w)$'s.

The $(2,2)$-block of $[A,X]$ is \[ -b_1 u_1^T + [A_1,X_{11}] + u_1 b_1^T = O. \] Thus $[A_1,X_{11}] = b_1 u_1^T - u_1 b_1^T $.
Since $b_1$ has just one nonzero entry,  the nonzero entries of  $b_1 u_1^T - u_1 b_1^T$ 
lie in its first row or first column. Thus  $[A_1, X_{11}](1)=O$.

So, $A_1$ and $X_{11}$ satisfy the induction hypothesis, and thus $X_{11}=O$ and 
$b_1 u_1^T - u_1 b_1^T= O$. But the first row of $b_1 u_1^T - u_1 b_1^T$ is a nonzero multiple
of $u_1^T$, we conclude that $u_1$ is the zero vector. Similarly, one can show that  each of $X_{22}, X_{33}, \ldots, X_{kk}, u_2, u_3, \ldots, u_k$ is zero.

Consider the $(r+1,s+1)$-block of $[A,X]$, where $r\neq s$. By (b),  $A_r X_{rs}=X_{rs}A_s$. Since $A$ has the  Duarte-property with respect to vertex $1$,  $A_r$ and $A_s$ have no common eigenvalue. Since $A_r X_{rs}=X_{rs}A_s$, and $A_r$ and $A_s$ do not have a common eigenvalue, $X_{rs}=O$ \cite[Lemma 1.1 (a)]{ms}. This holds for all $r$ and $s$. Thus $X=O$.
\end{proof}

\section{The $\lambda-\mu$ skew-symmetric SIEP for trees}\label{sectiontree}
In this section we formulate the $\lambda-\mu$ skew-symmetric SIEP for the class of NEB trees and provide a solution for it. Recall that, the characteristic polynomial of a real skew-symmetric matrix is a real polynomial. So, all of its eigenvalues occur in conjugate pairs. So, one of the necessary assumptions for this problem to have a solution is that the prescribed eigenvalues to come as conjugate pairs. The following theorem is our main theorem for this section and we provide the proof after mentioning some preliminary results. 
\begin{theorem}\label{main}
Let $T$ be a tree on $n$ vertices $1, 2,\ldots,n$ with $n\geq 2$. Let \[\lambda_1< \mu_1< \lambda_2<\cdots <\mu_{n-1}< \lambda_n\] be $2n-1$ real numbers such that \[\lambda_j=-\lambda_{n+1-j},\] for all $j=1,\ldots,n$, and \[\mu_k=-\mu_{n-k},\] for all  $k=1,\ldots,n-1$. If $T$ is NEB at a vertex $v$, then there exists a skew-symmetric matrix $A$ in $S^-(T)$ with eigenvalues $\i \lambda_1,\i \lambda_2, \ldots, \i \lambda_n$ such that the eigenvalues of $A(v)$ are $\i \mu_1,\i \mu_2, \ldots, \i \mu_{n-1}$.
\end{theorem}

If $p(x)$ is the characteristic polynomial of a real skew-symmetric matrix of order $n$, then the coefficient of $x^{n-k}$ is zero, for odd $k$. The lemma below shows that such polynomials behave rather nicely on the  imaginary axis of the complex plane. In particular, they map the imaginary axis either to itself or to the real axis. We will use this fact later to show that certain functions have zeros on the imaginary axis.
\begin{lemma}\label{polynomial}
Let $p(x)= \displaystyle\sum_{j=0}^{n} a_{j} x^{n-j}$ be a real polynomial where $a_j=0$ for all odd $j$. Then 
\begin{enumerate}
\item[\rm (a)] if $n$ is even, then $p(\i a)$ is real for any real number $a$, and
\item[\rm (b)] if $n$ is odd, then $p(\i a)$ is purely imaginary for any real number $a$.
\end{enumerate}
\end{lemma}
\begin{proof}
It follows from the fact that for any nonzero real number $a$, $(\i a)^k$ is a real number, for any even integer $k$, and it is purely imaginary, for any odd integer $k$. 
\end{proof}

The following lemma plays a key role in the proof of Theorem \ref{main}.
\begin{lemma}\label{pfd}
Let $\lambda_1< \mu_1< \lambda_2<\cdots <\mu_{n-1}< \lambda_n$ be $2n-1$ real numbers such that $\lambda_j=-\lambda_{n+1-j}$ for all $j=1,\ldots,n$, and $\mu_k=-\mu_{n-k}$ for all  $k=1,\ldots,n-1$. Let $f(x)=\prod_{j=1}^n(x-\i \lambda_j)$ and  $g(x)=\prod_{j=1}^{n-1}(x-\i\mu_j)$. Then
\begin{enumerate}
\item[\rm (a)]\label{pfda} the coefficient of $x^{n-1}$ in $f(x)$ and the coefficient of $x^{n-2}$ in $g(x)$ are zero, and
\item[\rm (b)]\label{pfdb} $\displaystyle\frac{f(x)}{g(x)}=x+\sum_{j=1}^{n-1}\frac{c_j}{x-\i \mu_j}$ for some $c_j>0$ where  $c_k=c_{n-k}$ for $k=1,\ldots,n-1$.
\end{enumerate}
\end{lemma}
\begin{proof}
(a) The coefficient of $x^{n-1}$ in $f(x)$ is $-\sum_{j=1}^n \i \lambda_j$. Since $\lambda_j=-\lambda_{n+1-j}$ for all $j=1,\ldots,n$, we have $-\sum_{j=1}^n \i \lambda_j=0$. Similarly the coefficient of $x^{n-2}$ in $g(x)$ is zero.

(b) Since all the roots of $g(x)$ are distinct, by the partial fraction decomposition we get 
\begin{equation}\label{partialeffovergee}
\displaystyle\frac{f(x)}{g(x)} = x+a + \sum_{j=1}^{n-1} \frac{c_j}{x-\i \mu_j}
\end{equation}
for some complex numbers $a,c_j$. Since the coefficient of $x^{n-1}$ in $f(x)$ and the coefficient of $x^{n-2}$ in $g(x)$ are zero, by direct division of $f(x)$ by $g(x)$, we have $a=0$, thus (\ref{partialeffovergee}) becomes
\begin{equation}\label{secondpartialeffovergee}
\displaystyle\frac{f(x)}{g(x)}=x+\sum_{j=1}^{n-1}\frac{c_j}{x-\i \mu_j}
\end{equation}
Multiplying both sides of (\ref{secondpartialeffovergee}) by $g(x)$ we get 
\begin{equation}\label{thislastequation}
f(x)=xg(x)+\sum_{j=1}^{n-1}\frac{c_jg(x)}{x-\i \mu_j}
\end{equation}
Plugging $x=\i \mu_k$ in (\ref{thislastequation}), we get 
\[f(\i \mu_k)=\displaystyle\sum_{j=1}^{n-1}\frac{c_jg(\i \mu_k)}{\i \mu_k- \i \mu_j} = \i^{n-2} c_k \prod_{\substack{1\leq j\leq n-1\\j\neq k}}(\mu_k-\mu_j).\]
But by the definition of $f$, $f(\i \mu_k)=\prod_{j=1}^n(\i \mu_k- \i \lambda_j) = \i^n \prod_{j=1}^n(\mu_k-\lambda_j)$. Thus 
\begin{equation}\label{ceekay}
c_k = \displaystyle\frac{\i^n \displaystyle\prod_{j=1}^n(\mu_k-\lambda_j)}{\i^{n-2} \displaystyle\prod_{\substack{1\leq j\leq n-1\\j\neq k}}(\mu_k-\mu_j)}=-\displaystyle\frac{\displaystyle\prod_{j=1}^n(\mu_k-\lambda_j)}{\displaystyle\prod_{\substack{1\leq j\leq n-1\\j\neq k}}(\mu_k-\mu_j)}.
\end{equation}
Since $\lambda_1< \mu_1< \lambda_2<\cdots <\mu_{n-1}< \lambda_n,$ in (\ref{ceekay}) the product in the numerator has exactly $n-k$ negative terms and the product in the denominator has exactly $n-k-1$ negative terms. Thus $c_k$ is a positive real number.

Now we show that  $c_k=c_{n-k}$ for $k=1,\ldots,n-1$. Observe that in (\ref{ceekay}), $\lambda_j=-\lambda_{n+1-j}$, for $j=1,\ldots,n$, and $\mu_k=-\mu_{n-k}$, for $k=1,\ldots,n-1$.  There are two cases according to the parity of $n$. We consider the case when $n$ is even, and it follows similarly when $n$ is odd. Let $n=2l$ for some positive integer $l$. Suppose that $n-k=2l-k\geq l+1$. Since $\mu_{n-k} = - \mu_k$,
\[ c_{n-k}=-\displaystyle\frac{\displaystyle\prod_{j=1}^{n}(\mu_{n-k}-\lambda_j)}{\displaystyle\prod_{\substack{1\leq j\leq n-1\\j\neq n-k}}(\mu_{n-k}-\mu_j)}
=-\displaystyle\frac{\displaystyle\prod_{j=1}^{n}(-\mu_{k}-\lambda_j)}{\displaystyle\prod_{\substack{1\leq j\leq n-1\\j\neq n-k}}(-\mu_{k}-\mu_j)}. \]
Break each of the products in the right hand side above into two halves to get 
\[  c_{n-k}= -\displaystyle\frac{\displaystyle\prod_{j=1}^{l}(-\mu_{k}-\lambda_j)\displaystyle\prod_{j=l+1}^{2l}(-\mu_{k}-\lambda_j)}{\displaystyle\prod_{1\leq j\leq l}(-\mu_{k}-\mu_j)\prod_{\substack{l+1\leq j\leq 2l-1\\j\neq 2l-k}}(-\mu_{k}-\mu_j)}. \]
Since $\lambda_{n+1-k} = - \lambda_k$, by reordering the products we have
\[  c_{n-k}= -\displaystyle\frac{\displaystyle\prod_{j=l+1}^{2l}(-\mu_{k}+\lambda_j)\displaystyle\prod_{j=1}^{l}(-\mu_{k}+\lambda_j)}{\displaystyle\prod_{l+1\leq j\leq 2l-1}(-\mu_{k}+\mu_j)\prod_{\substack{1\leq j\leq l\\j\neq k}}(-\mu_{k}+\mu_j)}.\]
Factor a $-1$ from each term of each product 
\[c_{n-k}=-\displaystyle\frac{ \left[ (-1)^l\displaystyle\prod_{j=l+1}^{2l}(\mu_{k}-\lambda_j)\right] \left[ (-1)^l\displaystyle\prod_{j=1}^{l}(\mu_{k}-\lambda_j)\right] }{\left[(-1)^{l-1}\displaystyle\prod_{l+1\leq j\leq 2l-1}(\mu_{k}-\mu_j)\right]\left[(-1)^{l-1}\displaystyle\prod_{\substack{1\leq j\leq l\\j\neq k}}(\mu_{k}-\mu_j)\right]},\]
Now multiply all $-1$'s to get
\[c_{n-k}=-\displaystyle\frac{\displaystyle\prod_{j=l+1}^{2l}(\mu_{k}-\lambda_j)\displaystyle\prod_{j=1}^{l}(\mu_{k}-\lambda_j)}{\displaystyle\prod_{l+1\leq j\leq 2l-1}(\mu_{k}-\mu_j)\displaystyle\prod_{\substack{1\leq j\leq l\\j\neq k}}(\mu_{k}-\mu_j)}=-\displaystyle\frac{\displaystyle\prod_{j=1}^{2l}(\mu_{k}-\lambda_j)}{\displaystyle\prod_{\substack{1\leq j\leq 2l-1\\j\neq k}}(\mu_{k}-\mu_j)}=c_k.\]

If $n-k=2l-k\leq l$, it can be proved similarly. Also, the case for $n$ odd follows similarly.
\end{proof}

The following Lemma may be proved using techniques similar to that of the proof of Lemma 2 in \cite{parter}. A similar lemma is used in \cite{duarte} in the case of Hermitian matrix $A$  whose graph is a tree.
\begin{lemma}\label{lastlemma}
Let $T$ be a tree on $n$ vertices $1,\ldots,n$, with $1\leq v\leq n$. Let $A=[a_{k,l}]$ be a skew-symmetric matrix such that $G(A)=T$. Then 
\[\displaystyle\frac{C_{A}(x)}{C_{A(v)}(x)} = x + \displaystyle\sum_{j\in N(v)} a_{vj}^2 \displaystyle\frac{C_{A_{j'}(v)}(x)}{C_{A_{j}(v)}(x)}.\]
\end{lemma}
For further details on the characteristic polynomial of the skew-adjacency matrix of a graph see \cite{gregory}. Now we have all the tools to prove the main theorem of this section (Theorem \ref{main}). Recall that, we want to prove if $T$ is a tree on $n$ vertices, $\lambda_1 < \mu_1 < \lambda_2 < \cdots < \mu_{n-1} < \lambda_n$ are $2n-1$ real numbers such that $\lambda_j=-\lambda_{n+1-j}$ and $\mu_k=-\mu_{n-k}$, and if $T$ is NEB at a vertex $v$, then there exists a real skew-symmetric matrix $A$ in $S^-(T)$ with eigenvalues $\i \lambda_1,\i \lambda_2, \ldots, \i \lambda_n$ such that the eigenvalues of $A(v)$ are $\i \mu_1,\i \mu_2, \ldots, \i \mu_{n-1}$.

\begin{proof}[Proof of Theorem $\ref{main}$] We prove this by induction on $n$.\\
For $n=2$, $\mu_1=0$ and the desired matrix is $\left[\begin{array}{cc}0&\lambda_1\\-\lambda_1&0\end{array}\right]$.

Now assume that the result is true for all $p<n$. There are two cases for $n$, we prove the result when $n$ is even, the case when $n$ is odd follows similarly.

Without loss of generality assume that, $v=1$ and $ N(v)=\{2,\ldots, m\}$. For each $j\in N(v)$ let $g_j(x)$ be the monic polynomial such that $\deg(g_j)=|T_j(v)|$ and roots of $g_j$ are $0$ or complex conjugate purely imaginary numbers. Let $g=g_2\cdots g_m$ such that $g(x)=\prod_{j=1}^{n-1}(x-\i\mu_j)$, this is possible because $T$ is NEB at $v$. Let $f(x)=\prod_{j=1}^n(x-\i \lambda_j)$. By the partial fraction decomposition we get 
\[\displaystyle\frac{f(x)}{g(x)}=x+a+\sum_{j\in N(v)} y_j\frac{h_j(x)}{g_j(x)},\] 
for some complex numbers $a,y_2,\ldots,y_m$ and unique monic polynomials $h_2,\ldots,h_m$ with $\deg h_j<\deg g_j$, for each $j\in N(v)$. By Lemma \ref{pfd}{\rm(a)}, we get $a=0$ and then
\begin{equation}\label{effovergee}
\displaystyle\frac{f(x)}{g(x)}=x+\sum_{j\in N(v)} y_j\frac{h_j(x)}{g_j(x)}.
\end{equation}
Now we use the following claims which are proved at the end of this proof.
\begin{enumerate}
	\item[Claim 1.] For each $j\in N(v)$, $y_j$ is a positive real number.
	\item[Claim 2.] The polynomial $h_j(x)$ is a real polynomial in $x$ for each $j=2,\ldots,m$. Moreover, if $\deg(h_j)$ is even, then the coefficients of the odd powers of $x$ in $h_j(x)$ are zero, and if $\deg(h_j)$ is odd, then the coefficients of the even powers of $x$ in $h_j(x)$ are zero.
	\item[Claim 3.] For each $j=2,\ldots,m$, $h_j$ has $\deg g_j-1$ distinct roots and the roots of $h_j$ strictly interlace the roots of $g_j$.
\end{enumerate}

Since $T_j(v)$ is a NEB tree at $j$, by the induction hypothesis, there exists $B_j\in S^-(T_j(v))$ with characteristic polynomial $g_j$ such that $h_j$ is the characteristic polynomial of $B_{j'}(v)$.

Now define an $n\times n$ skew-symmetric matrix $A=[a_{j,k}]$ such that $a_{v,j}=-a_{j,v}=\sqrt{y_j}$, $A_{j}(v)=B_j$ for $j\in N(v)$, and  all other entries of $A$ are zero. 

By construction of $A$, $g$ is the characteristic polynomial of $A(v)$. Finally by Lemma \ref{lastlemma}, $A$ has eigenvalues $\i \lambda_1,\i \lambda_2, \ldots, \i \lambda_n$.

\begin{proof}[Proof of Claim $1$]
There are two cases according to the parity of $\deg g_j$. First, let $\deg g_j=2l$ for some positive integer $l$. Let 
\[g_j(x)=\prod_{r=1}^l(x-\i \mu_{k_r})(x- \i \mu_{n-{k_r}}).\] 
Thus 
\begin{equation}\label{whyechovergee}
y_j\displaystyle\frac{h_j(x)}{g_j(x)}=\sum_{r=1}^l\frac{c_{k_r}}{x-\i \mu_{k_r}}+\frac{c_{n-k_{r}}}{x-\i \mu_{n-k_r}},
\end{equation}
for some complex numbers $c_{k_1},\ldots, c_{k_{l}},c_{n-k_1},\ldots, c_{n-k_{l}}$. By Lemma \ref{pfd}(b), $c_{k_1},\ldots, c_{k_{l}},c_{n-k_1},\ldots, c_{n-k_{l}}$ are positive real numbers. Note that from (\ref{whyechovergee}) we have $y_j=\sum_{r=1}^lc_{k_r}+c_{n-k_{r}}$. Since $c_{k_1},\ldots c_{k_l}$ are positive real numbers, $y_j>0$. When $\deg g_j=2l+1$, for some positive integer $l$,  the only other factor of $g_j$ is $x$, hence
\begin{equation}\label{whyechovergeeodd}
y_j\displaystyle\frac{h_j(x)}{g_j(x)}=\frac{c_0}{x}+\sum_{r=1}^l\frac{c_{k_r}}{x-\i \mu_{k_r}}+\frac{c_{n-k_{r}}}{x-\i \mu_{n-k_r}},
\end{equation}
and the claim follows similarly.
\end{proof}

\begin{proof}[Proof of Claim $2$]
Recall that $c_k=-c_{n-k}$ and $\mu_k=-\mu_{n-k}$ for $k=1,\ldots,n-1$. Assume that $\deg(g_j)$ is even. From (\ref{whyechovergee}) we have
\begin{align}
y_j\displaystyle\frac{h_j(x)}{g_j(x)}&=\sum_{r=1}^l\frac{c_{k_r}}{x-\i \mu_{k_r}}+\frac{c_{k_{r}}}{x+\i \mu_{k_r}} = \sum_{r=1}^l\frac{2c_{k_r}x}{x^2+\mu_{k_r}^2} \\ &= \displaystyle\frac{x \displaystyle\sum_{r=1}^{l}2c_{k_r}\displaystyle\prod_{\substack {s=1 \\ s\neq r}}^{l} (x^2+\mu_{k_s}^2)}{\displaystyle\prod_{r=1}^{l} (x^2+\mu_{k_r}^2)}  = \frac{x}{g_{j}(x)} \displaystyle\sum_{r=1}^{l}2c_{k_r}\displaystyle\prod_{\substack {s=1 \\ s\neq r}}^{l} (x^2+\mu_{k_s}^2)
\end{align}
Hence, \[h_j(x) = \dfrac{x}{y_j} \displaystyle\sum_{r=1}^{l}2c_{k_r} \displaystyle\prod_{\substack {s=1 \\ s\neq r}}^{l} (x^2+\mu_{k_s}^2).\]
Since $y_j,c_{k_1},\ldots,c_{k_l}$ are real numbers, $h_j(x)$ is a real polynomial of odd degree, and the coefficients of the even powers of $x$ in $h_j(x)$ are zero. Similarly, when $\deg(g_j)$ is odd, $h_j$ is a real polynomial of even degree and the coefficients of the odd powers of $x$ in $h_j(x)$ are zero.
\end{proof}

\begin{proof}[Proof of Claim $3$]
Let $\mu_ri$ be the smallest root of $g_j$ and $\mu_{r+p}i$ be the second smallest root of $g_j$. Then from (\ref{effovergee}) we have
\begin{equation}\label{effofmuareeye}
f(\mu_ri)=g(\mu_ri)\cdot y_j\displaystyle\frac{h_j(\mu_ri)}{g_j(\mu_ri)}=y_jh_j(\mu_ri)\sum_{\substack{t=2\\t\neq j}}^mg_t(\mu_ri).
\end{equation}  
Similarly from (\ref{effovergee}) we have
\begin{equation}\label{effofmuareeyepluspee}
f(\mu_{r+p}i)=y_jh_j(\mu_{r+p}i)\sum_{\substack{t=2\\t\neq j}}^mg_t(\mu_{r+p}i).\end{equation} 
Let $R_j$ be the set of the all roots of $\dfrac{g(x)}{g_j(x)}$ for each $j=2,\ldots,m$. Then by (\ref{effofmuareeye}) and (\ref{effofmuareeyepluspee}) we have 
\begin{equation}\label{secondeffofmuareeye}
f(\mu_ri)=y_jh_j(\mu_ri)\sum_{\mu\notin R_j}(\mu_r-\mu)i^{n-n_j},
\end{equation}
\begin{equation}\label{secondeffofmuareeyepluspee}
f(\mu_{r+p}i)=y_jh_j(\mu_{r+p}i)\sum_{\mu\notin R_j}(\mu_{r+p}-\mu)i^{n-n_j}.
\end{equation}
We know that $f(x)=\prod_{k=1}^n(x-i\lambda_k)$. Then we have
\begin{equation}\label{thirdeffofmuareeye}
f(\mu_ri)=\displaystyle\prod_{k=1}^n(\mu_ri-\lambda_ki)=i^n\prod_{k=1}^n(\mu_r-\lambda_k),
\end{equation}
\begin{equation}\label{thirdeffofmuareeyepluspee}
f(\mu_{r+p}i)=\displaystyle\prod_{k=1}^n(\mu_{r+p}i-\lambda_ki)=i^n\prod_{k=1}^n(\mu_{r+p}-\lambda_k).
\end{equation}
Since $n$ is even, $f(\mu_ri)$ and $f(\mu_{r+p}i)$ are real numbers (If $n$ is odd, then $f(\mu_ri)$ and $f(\mu_{r+p}i)$ are purely imaginary numbers). If $p$ is odd, then there are $p$ $\lambda$'s between $\mu_r$ and $\mu_{r+p}$. Then by (\ref{thirdeffofmuareeye}) and (\ref{thirdeffofmuareeyepluspee}), $f(\mu_ri)$ and $f(\mu_{r+p}i)$ have the opposite signs. 
Now by direct counting of $\mu$'s, $\sum_{\mu\notin R_j}(\mu_r-\mu)$ and $\sum_{\mu\notin R_j}(\mu_{r+p}-\mu)$ have the same sign. Thus by (\ref{secondeffofmuareeye}) and (\ref{secondeffofmuareeyepluspee}), $h_j(\mu_ri)$ and $h_j(\mu_{r+p}i)$ have the opposite signs. Similarly when $p$ is even we can show that $h_j(\mu_ri)$ and $h_j(\mu_{r+p}i)$ have the opposite signs. First note that, by Claim 2 and Lemma \ref{polynomial}, $g_j$ is a real polynomial that maps $\i\mathbb{R}$ to either $\i\mathbb{R}$ or $\mathbb{R}$. Then by the Intermediate Value Theorem, $h_j$ has a purely imaginary root between each two consecutive roots of $g_j$. Thus $\deg h_j=\deg g_j-1$ and the roots of $h_j$ strictly interlace the roots of $g_j$.
\end{proof}
This completes the proof.
\end{proof}

By the construction of $A$ in the proof of the preceding theorem, it is clear that $A$ has the Duarte property with respect to vertex $v$.

\begin{corollary}\label{NEB2Duarte}
The matrix $A$ constructed in the proof of Theorem $\ref{main}$ has the Duarte property with respect to vertex $v$.
\end{corollary}

\begin{remark}
For a tree $T$, Theorem \ref{Duarte2NEB} shows that if a matrix $A \in S^- (T)$ has the Duarte property with respect to a vertex $v$, then $T$ is NEB at $v$. Conversely, Corollary \ref{NEB2Duarte} shows that if $T$ is NEB at a vertex $v$, then there is an $A \in S^- (T)$ which is Duarte with respect to $v$.
\end{remark}

\begin{example}\label{exampleP4}

Let $T$ be the path $P_4$ on four vertices $1$,$2$,$3$ and $4$ where vertex $4$ is a pendent vertex. Consider seven real numbers $-2<-1.5<-1<0<1<1.5<2$. Following the proof of Theorem \ref{main}, we will find a $4\times 4$ real skew-symmetric matrix $A$ such that $G(A)=P_4$ , the eigenvalues of $A$ are $\pm i,\pm 2i$, and the eigenvalues of $A(4)$ are $0,\pm 1.5i$. An approximation for such matrix is given below.

\begin{minipage}[b]{0.45\textwidth}
\[A \simeq \left[\begin{array}{cccc}
0 & 1.206045 & 0 & 0 \\
-1.206045 & 0 & 0.8918826 & 0 \\
0 & -0.8918826 & 0 & 1.658312 \\
0 & 0 & -1.658312 & 0
\end{array}\right]\]
\end{minipage} \begin{minipage}[c]{0.45\textwidth}
\begin{flushright}
\begin{tikzpicture}[scale=.8,colorstyle/.style={circle, draw=black!100,fill=black!100, thick, inner sep=0pt, minimum size=1.2 mm}]
\node (1) at (1,1)[colorstyle,label=left:$1$]{};
\node (2) at (1,0)[colorstyle,label=left:$2$]{};
\node (3) at (2,0)[colorstyle,label=right:$3$]{};
\node (4) at (2,1)[colorstyle,label=right:$4$]{};
\draw [thick] (1)--(2)--(3)--(4);
\node () at (1.5,-1) {$P_4$};
\end{tikzpicture}
\end{flushright}
\end{minipage}\\	

It is easy to check that $A$ has the Duarte property with respect to vertex $4$: Eigenvalues of $A$ are approximately $\pm \i, \pm 2 \i$, eigenvalues of $A(4)$ are approximately $0, \pm 1.5 \i$, eigenvalues of $A(\{4,3\})$ are approximately $\pm 1.206045\i$, and finally, and the eigenvalue of $A(\{4,3,2\})$ is $0$. They satisfy the strict interlacing inequality conditions in the definition of the Duarte property.
\end{example}

A {\it matching} in a graph $G$ is a set of vertex-disjoint edges. A {\it maximum matching} in $G$ is a matching with the maximum number of edges among all matchings in $G$. The matching number, denoted by $\match(G)$, is the number of edges in a maximum matching in $G$. The following observation shows that if a tree $T$ is NEB at a vertex, then $\match(T)$ is as large as possible.

\begin{obs}\label{matching} Suppose $T$  is a tree of order $n$. If $T$ is NEB at a vertex, then $\match(T)=\lfloor \frac{n}{2}\rfloor$. 
\end{obs}
\begin{proof}
Suppose that $T$ is NEB at a vertex. Then by Theorem \ref{main}, we can find some $A$ in $S^-(T)$ with distinct eigenvalues. Note that any $A$ in $S^-(T)$ has rank less than or equal to $2\match(T)$ \cite[Theorem 2.5]{skewrank}. If $\match(T)<\lfloor \frac{n}{2}\rfloor$, then for any $A$ in $S^-(T)$ the multiplicity of the eigenvalue $0$ of $A$ is at least $2$. Thus $\match(T)=\lfloor \frac{n}{2}\rfloor$. ​
\end{proof}

\section{A polynomial map and its Jacobian}\label{sectionjacobian}
For the remainder of the paper fix $T$ to be a NEB tree at vertex $n$. Assume $T$ has vertices $1,2,\ldots, n$ and edges $e_1=\{i_1, j_1\}$, \ldots, $e_{n-1}= \{i_{n-1}, j_{n-1}\}$, where $i_k < j_k$ for $k=1,\ldots,n-1$.  Let $x_1, x_2, \ldots, x_{n-1}$ be $n-1$ independent indeterminates, and set 
\[x=(x_1,x_2, \ldots, x_{n-1}).\]
Define $M(x)$ to be the matrix with $x_{k}$ in the $(i_k,j_k)$ and $-x_{k}$ in the $(j_k,i_k)$ positions $(k=1,2,\ldots, n-1)$, and zeros elsewhere.  
Set $N(x)= M(x)(n)$; that is, $N(x)$ is the principal submatrix obtained from $M(x)$ by deleting its last row and column. We denote these matrices  by $M$ and $N$ for short.

\begin{example}\label{exampleP4MN}
Consider the tree $T$ from Example \ref{exampleP4}. The adjacency matrix of $T$ is

\[\left[\begin{array}{cccccc}
0 & 1 & 0 & 0\\
1 & 0 & 1 & 0\\
0 & 1 & 0 & 1\\
0 & 0 & 1 & 0\\
\end{array}\right].\]

Thus

\[M=\left[\begin{array}{cccc}
0&x_1&0&0\\
-x_1&0&x_2&0\\
0&-x_2&0&x_3\\
0&0&-x_3&0
\end{array}\right], \qquad N=M(4)=\left[\begin{array}{ccc}
0&x_1&0\\
-x_1&0&x_2\\
0&-x_2&0\\
\end{array}\right].\]
\end{example}

Suppose that $t^n + c_{1}t^{n-1} + \cdots +c_{n-1} t + c_n$ and 
$t^{n-1} + d_{1} t^{n-2} + \cdots + d_{n-2}t + d_{n-1}$ are the characteristic polynomials of $M$ and $N$, respectively. We now define four polynomial maps associated with $M$ and $N$. \\

Let $G:\mathbb{R}^{n-1}\rightarrow \mathbb{R}^{2n-1}$ be the polynomial map defined by
\begin{equation}\label{GEE}
G(x)=\left( c_1, c_{2}, \dots ,c_{n}, d_{1}, d_{2}, \dots , d_{n-1} \right).
\end{equation}

Since $M$ and $N$ are skew-symmetric matrices, $c_i$ and $d_i$ are zero for odd $i$. So, the function $G$ is mapping $\mathbb{R}^{n-1}$ to an $n-1$ dimensional subspace of $\mathbb{R}^{2n-1}$. So by restricting the codomain of $G$ we define a function $g:\mathbb{R}^{n-1}\rightarrow \mathbb{R}^{n-1}$ as follows.

\begin{equation}\label{gee}
g(x)=\left( c_2, c_4, \dots, d_{2}, d_{4}, \dots \right).
\end{equation}
 
The goal is to show that for a tree $T$ on $n$ vertices, and a matrix $A\in S^- (T)$ with the Duarte property with respect to the vertex $n$, the Jacobian of $g$ evaluated at the upper-triangular nonzero entries of $A$ is nonsingular. This enables us to use the Implicit Function Theorem, in order to perturb the zero entries of $A$, particularly making them nonzero, and to adjust the nonzero entries to obtain a new matrix $\widehat{A}$, such that  the characteristic polynomials of $\widehat{A}$ and $\widehat{A}(n)$ are equal to those of  $A$ and $A(n)$, respectively. That is, the graph of the matrix $\widehat{A}$ is a supergraph of $T$, and $\widehat{A}$ and $\widehat{A}(n)$ have the same eigenvalues as $A$ and $A(n)$, respectively.

It is not easy to show that the Jacobian of $g$ is nonsingular at some point. So, we introduce the following functions.\\

Let $F: \mathbb{R}^{n-1} \rightarrow \mathbb{R}^{2n-1}$ be the polynomial map defined by 
\begin{equation}
F(x) =  \left( \frac{\tr M}{2},\frac{\tr M^2}{4},\ldots ,\frac{\tr M^n}{2n},\frac{\tr N}{2},\frac{\tr N^2}{4},\ldots ,
\frac{\tr N^{n-1}}{2(n-1)} \right).
\end{equation}

Since $M,N$ are skew-symmetric matrices, for each $k$ we have $\tr M^{2k-1} = \tr N^{2k-1} = 0$, for all $x$. So by restricting the codomain of $F$ we define a function $f:\mathbb{R}^{n-1}\rightarrow \mathbb{R}^{n-1}$ as follows.
\begin{equation}\label{eff}
f(x) = \begin{cases}  \left( \frac{\tr M^2}{4},\frac{\tr M^4}{8}\ldots ,\frac{\tr M^{2m}}{4m},\frac{\tr N^2}{4},\frac{\tr N^4}{8}, \ldots , \frac{\tr N^{2(m-1)}}{4(m-1)} \right) & \mbox{ if } n=2m,\\ \\
\left( \frac{\tr M^2}{4},\frac{\tr M^4}{8}\ldots ,\frac{\tr M^{2m}}{4m},\frac{\tr N^2}{4},\frac{\tr N^4}{8}, \ldots , \frac{\tr N^{2m}}{4m} \right) & \mbox{ if } n=2m+1.
\end{cases}
\end{equation}

\begin{example}\label{exampleP4MNEff}
Consider the matrices $M$ and $N$ from Example \ref{exampleP4MN}. Then 

\[ g(x_1,x_2,x_3)= \left( x_1^2 + x_2^2 + x_3^2, x_1^2 x_3^2, x_1^2 + x_2^2 \right),\]
and 
\[f(x_1,x_2,x_3)= \left(-\frac{x_{1}^{2} +   x_{2}^{2} + x_{3}^{2}}{2}, \frac{(x_{1}^{4} + x_{2}^{4} + x_{3}^{4}) + 2 x_{2}^{2} ( x_{1}^{2} + x_{3}^{2} )}{4}, -\frac{x_1^2 + x_2^2}{2} \right).\]
\end{example}

Next, we give a closed formula for the Jacobian matrix of $f$ evaluated at a certain point, and then we show that the above Jacobian matrix is nonsingular whenever $A$ has the Duarte-property with respect to $n$. As it is mentioned in Section 3 of \cite{ms}, note that by Newton's identities, there is an infinitely differentiable,  invertible $h: \mathbb{R}^{n-1} \rightarrow \mathbb{R}^{n-1}$ such that $g\circ h=f$.  Thus, the Jacobian matrix of $f$ at a point $x$ is nonsingular if and only if the Jacobian matrix of $g$ at $h(x)$ is nonsingular.

\begin{lemma}
\label{derivatives}
Let $k$ be a positive even integer and $(i,j)$ be  a nonzero position of $M$ with corresponding variable $x_{t}$. Then 
\begin{itemize}
\item[\rm (a)] 
$\displaystyle \frac{\partial}{\partial x_{t}} \left( \tr M^k \right)= - 2k \left( M^{k-1} \right)_{ij}$, and
\item[\rm (b)] 
 $ \displaystyle
 \frac{\partial}{\partial x_{t}}\left( \tr {N}^k \right)=  \left\{ \begin{array}{ll} - 2k \left( N^{k-1} \right)_{ij} & \mbox{ if neither $i$ nor $j$ is $n$}
\\
0 & \mbox{ otherwise.} \end{array} \right.
$
\end{itemize}
\end{lemma}
\begin{proof}
Without loss of generality, assume $i < j$. Let $E_{ij}$ be the matrix (of  appropriate size) with a $1$ in position $(i,j)$ and 
$0$s elsewhere. First, note that $$\frac{\partial}{\partial x_{t} } M = E_{ij}-E_{ji},$$ thus
\begin{align*}
\frac{\partial}{\partial x_{t} }\left( \tr (M^k) \right) &= \sum_{\ell=0}^{k-1}\tr\left( M^{\ell} \cdot  \frac{\partial}{\partial x_t} M \cdot M^{k-\ell-1} \right) & \mbox{{\footnotesize (by chain rule)}}\\
&= \sum_{\ell=0}^{k-1}\tr\left( M^{k-1} \cdot \frac{\partial}{\partial t} M  \right) & \mbox{{\footnotesize (\mbox{since} $\tr(AB)=\tr(BA)$ \mbox{for all matrices} $A$ \mbox{and} $B$})}\\
&= k \tr \left( M^{k-1}(E_{ij}-E_{ji}) \right)\\ 
&= k \left( (M^{k-1})_{ji} - (M^{k-1})_{ij} \right)\\ 
&= - 2k (M^{k-1})_{ij}. & \mbox{{\footnotesize (\mbox{since} $M^{k-1}$ is skew-symmetric)}}
\end{align*}
A similar argument works for $N$, provided we note that if $i$ or $j$ equals $n$, then $N$ does not contain $x_t$, and consequently $\displaystyle\frac{\partial}{\partial x_t} N= 0$. 
\end{proof}

We will use the following notations in the rest of this paper.

\begin{notation} \label{widetildenotation}
Given any $(n-1)\times (n-1)$ matrix $W$, we set
$$
\widetilde{W} = \left[ \begin{array}{c|c} W & \begin{array}{c} 0 \\  \vdots \\ 0 \end{array} \\ \hline
0  \; \cdots \; 0 & 0 \end{array} \right]
.$$
\end{notation}

\begin{notation}\label{Jacobiannotation}
Given a matrix $A=[a_{i,j}]\in  S^-(T)$ we denote by $ \jac(f) \;{\rule[-2.1mm]{.1mm}{6mm}}_{A}$ the matrix obtained from $\jac(f)$ by evaluating at $(x_1, \ldots, x_{n-1})$ where $x_k$ equals the corresponding entry of $A$, for $k=1,2,\ldots, n-1$. 
\end{notation}

Using Notations \ref{widetildenotation} and \ref{Jacobiannotation}, Lemma \ref{derivatives} implies the following. 
\begin{corollary}
\label{Jacobian}
Let $T$  be a tree defined as above  on $n$ vertices, and $A \in S^-(T)$, and let $B=A(n)$. Then

\begin{align*} -\jac(f) \;{\rule[-3.6mm]{.1mm}{8mm}}_{A}&= 
\left[ \begin{array}{cccc}
&&\\
A_{i_1j_1}&A_{i_2j_2}&\cdots& A_{i_{n-1} j_{n-1}}\\
A^3_{i_1j_1}  &A^3_{i_2j_2}  & \cdots & A^3_{i_{n-1}j_{n-1}}\\
\vdots&\vdots&\ddots&\vdots\\
A^{n-1}_{i_1j_1} & A^{n-1}_{i_2j_2} & \cdots & A^{n-1}_{i_{n-1}j_{n-1}}\\ &&\\\hline &&\\
\widetilde{B}_{i_1j_1} & \widetilde{B}_{i_2j_2} & \cdots&\widetilde{B}_{i_{n-1} j_{n-1}}\\
\widetilde{B}^3_{i_1j_1} &\widetilde{B}^3_{i_2j_2} & \cdots & \widetilde{B}^3_{i_{n-1}j_{n-1}}\\
\vdots&\vdots&\ddots&\vdots\\
\widetilde{B}^{n-3}_{i_1j_1} &\widetilde{B}^{n-3}_{i_2j_2} & \cdots & \widetilde{B}^{n-3}_{i_{n-1}j_{n-1}}
\end{array} \right] \mbox{ or } \left[ \begin{array}{cccc}
&&\\
A_{i_1j_1}&A_{i_2j_2}&\cdots& A_{i_{n-1} j_{n-1}}\\
A^3_{i_1j_1}  &A^3_{i_2j_2}  & \cdots & A^3_{i_{n-1}j_{n-1}}\\
\vdots&\vdots&\ddots&\vdots\\
A^{n-2}_{i_1j_1} & A^{n-2}_{i_2j_2} & \cdots & A^{n-2}_{i_{n-1}j_{n-1}}\\ &&\\\hline &&\\
\widetilde{B}_{i_1j_1} & \widetilde{B}_{i_2j_2} & \cdots&\widetilde{B}_{i_{n-1} j_{n-1}}\\
\widetilde{B}^3_{i_1j_1} &\widetilde{B}^3_{i_2j_2} & \cdots & \widetilde{B}^3_{i_{n-1}j_{n-1}}\\
\vdots&\vdots&\ddots&\vdots\\
\widetilde{B}^{n-2}_{i_1j_1} &\widetilde{B}^{n-2}_{i_2j_2} & \cdots & \widetilde{B}^{n-2}_{i_{n-1}j_{n-1}}
\end{array} \right]. \end{align*}
The former happens when $n$ is even, and the latter happens when $n$ is odd.
\end{corollary}

\begin{theorem}\label{JAC}
Let $T$ be an NEB tree at vertex $n$. Let matrices $A$, $B$, and function $f$ (the function defined in terms of the traces of $M$ and $N$, equation $(\ref{eff})$) be defined as above. If $A$ has the Duarte-property with respect to vertex $n$,  then $\jac(f)\atA$ is nonsingular.
\end{theorem}

\begin{proof}
Note that $\jac(f)\atA$ is nonsingular if and only if  the only vector  
$\alpha=(\alpha_1,\alpha_2,\dots,\alpha_{n-1})^T$ such that $\alpha^T  \jac(f)\atA = (0, \ldots, 0)$ is the zero-vector.

Let $\jac(f)_k$ denote the $k^{\tiny\mbox{th}}$ row of $- \jac(f)\atA$. So $\alpha^T \jac(f)\atA= \sum_{k=1}^{n-1} \alpha_k \jac(f)_k$.  There are two cases:

\begin{enumerate}
	\item[Case 1.] $n=2m$. \\
	For $\ell = 1,\ldots, n-1$, the $\ell$-th entry in $\alpha^T \jac(f)\atA$ is the $(i_{\ell}, j_{\ell})$-entry of  \[ \sum_{r=1}^{m} \alpha_r A^{2r-1} + \sum_{r=1}^{m-1} \alpha_{m+r} \widetilde{B}^{2r-1}.\] 
In this case let $p(x) = \sum_{r=1}^{m} \alpha_r x^{2r-1}$, and $q(x) = \sum_{r=1}^{m-1} \alpha_{m+r} x^{2r-1}$.
	\item[Case 2.] $n=2m+1$.\\
	For $\ell = 1,\ldots, n-1$,  the $\ell$-th entry in $\alpha^T \jac(f)\atA$ is the $(i_{\ell}, j_{\ell})$-entry of  \[ \sum_{r=1}^{m} \alpha_r A^{2r-1} + \sum_{r=1}^{m} \alpha_{m+r} \widetilde{B}^{2r-1}.\] 
In this case let $p(x) = \sum_{r=1}^{m} \alpha_r x^{2r-1}$, and $q(x) = \sum_{r=1}^{m} \alpha_{m+r} x^{2r-1}$.
\end{enumerate} 

Let 
\[X= p(A) + \widetilde{q(B)},\]
where $\widetilde{q(B)}$ is constructed from the matrix $q(B)$ by padding it with a zero row and a zero column, as in Notation \ref{widetildenotation}. Note that since $X$ only involves the odd powers of skew-symmetric matrices, it is skew-symmetric. Also, note that the columns of $\jac(f)\atA$ correspond to the nonzero positions of $A$. Thus, in either case $\alpha^T\jac(f)\atA$ is the zero vector if and only if $X_{ij} = 0$ where $A_{ij}$ is nonzero. That is, $X$ satisfies $X \circ A=O$ and $X\circ I=O$. So, in order to show $\jac(A)$ is nonsingular we will show that $p(x)$ and $q(x)$ are zero polynomials. 

Observe that $[A,p(A)]=O$, hence $[A,X]= [A,\widetilde{q(B)}]$.
Also, note that since $A(n)=B$, $[A,\widetilde{q(B)}](n)=O$.
Thus, $[A,X](n)=O$, and by Lemma \ref{cent} we conclude that $X=O$. The rest of the proof is similar to that of Theorem 3.3 of \cite{ms}. $X=O$ implies that $p(A)=-\widetilde{q({B})}$. Let $Y:=p(A)=-\widetilde{q{({B})}}$, then $AY=Ap(A)$. 
We want to show that $Y=O$. Multiplying $A$ and $p(A)$ we get
\begin{align*}
Ap(A) &= -A(\widetilde{q(B)}) =
-\left[\begin{array}{c|c}
\mbox{\large{$B$}} & 
\begin{array}{c} \ast  \\ \vdots  \\ \ast  \end{array} \\ \hline
 \ast  \; \cdots \; \ast  & \ast  \end{array} 
\right]
\left[\begin{array}{c|c}
\mbox{\large{$q(B)$}} & 
\begin{array}{c} 0 \\ \vdots \\ 0 \end{array} \\ \hline
 0 \; \cdots \; 0 & 0 \end{array} 
\right]\\
&= \left[\begin{array}{c|c}
\mbox{\large{$-Bq(B)$}} & 
\begin{array}{c} 0 \\ \vdots \\ 0 \end{array} \\ \hline
 \ast  \; \cdots \; \ast  & 0 \end{array} 
\right],
\end{align*}
and 
\begin{align*}
p(A)A &= -(\widetilde{q(B)}) A =
-\left[\begin{array}{c|c}
\mbox{\large{$q(B)$}} & 
\begin{array}{c} 0 \\ \vdots \\ 0 \end{array} \\ \hline
 0 \; \cdots \; 0 & 0 \end{array} 
\right]
\left[\begin{array}{c|c}
\mbox{\large{$B$}} & 
\begin{array}{c} \ast  \\ \vdots \\ \ast  \end{array} \\ \hline
 \ast  \; \cdots \; \ast  & \ast  \end{array} 
\right]\\
&= \left[\begin{array}{c|c}
\mbox{\large{$-q(B)B$}} & 
\begin{array}{c}  \ast  \\ \vdots \\ \ast  \end{array} \\ \hline
0 \; \cdots \; 0 & 0 \end{array} 
\right].
\end{align*}

Since  $Ap(A)=p(A)A$, the last row of $Ap(A)$ is zero and the last column of $Ap(A)$ is zero. Thus, $Ap(A)=-\widetilde{q({B})}\widetilde{B}=p(A)\widetilde{B}$. That is, $AY=Y\widetilde{B}$. Hence, either $Y=O$, or $A$ and $\widetilde{B}$ have a common eigenvalue \cite[Lemma 1.1 (a)]{ms}. If $Y=O$ we are done. Otherwise, since $A$ and $B$ have no common eigenvalue, $A$ and $\widetilde{B}$ both have an eigenvalue $0$ of multiplicity one. 

Let $Y_j$ be a the $j$-th column of $Y$, and assume that it is nonzero. Observe that the last entry of $Y_j$ is $0$.   By Lemma 1.1 (b) of \cite{ms}, $Y_j$ is a generalized eigenvector of $A$ corresponding to $0$.  Since $A$ is skew-symmetric with distinct eigenvalues, $Y_j$ is an eigenvector of $A$ corresponding to $0$. This implies that the vector $Y_j(n)$ is a nonzero eigenvector of $B$ corresponding to $0$. This leads to the contradiction that $A$ and $B$ have a common eigenvalue. Thus $Y=O$.
  
  Since $Y=O$, $p(A)=O$ and $q(B)=O$. Note that  $p(x)$ is a polynomial of degree at most $n-1$. Since $A$ has $n$ distinct eigenvalues, its minimal polynomial has degree $n$. Thus $p(x)$ is the zero polynomial. Similarly, $q(x)$ is the zero polynomial. So $\jac(f)\atA$ is nonsingular. 
\end{proof}

\section{The $\lambda-\mu$ skew-symmetric SIEP for connected graphs with a NEB spanning tree}\label{sectionconnected}
We have shown that for any NEB tree $T$ at a vertex $v$, and sets of `generic' purely imaginary numbers, one can find a real skew-symmetric matrix $A$ with graph $T$ and the spectra given by the specified purely imaginary numbers. Furthermore, we showed that $A$ is a `generic' solution. Now, we are going to use the Implicit Function Theorem (see \cite{kran02}) to find a solution $\widehat{A}$ where $G(\widehat{A})$ is a supergraph of $T$.

\begin{theorem}\label{IFT}
Let $F: \mathbb{R}^{s+r} \rightarrow \mathbb{R}^s$ be a continuously differentiable
 function on an open subset $U$ of 
 $\mathbb{R}^{s+r}$
defined by $$F(x,y)=(F_1(x,y), F_2(x,y), \ldots, F_s(x,y)),$$
 where $x=(x_1, \ldots, x_s) \in \mathbb{R}^s$ and $y \in \mathbb{R}^r$. 
 Let $(a,b)$ be an element of $U$
 with $a\in \mathbb{R}^s$ and $b\in \mathbb{R}^r$,  and $c$ be an element of $\mathbb{R}^s$ such that 
  $F(a,b)=c$.
  If $$ \left[\frac{\partial F_i}{\partial x_{j}}\;{\rule[-3.6mm]{.1mm} {8mm}}_{(a,b)} \right] $$
  is nonsingular, then there exist an open neighborhood $V$  containing $a$
 and an open neighborhood $W$ containing $b$ 
 such that  $V \times W \subseteq U$ and 
   for each $y\in W$ there is an $x\in V$ with 
$F(x,y)=c$.
\end{theorem}

\begin{theorem}\label{mainconnected}
Let $G$ be a connected graph on $n$ vertices $1, 2,\ldots,n$ with $n\geq 2$. Let \[\lambda_1< \mu_1< \lambda_2<\cdots <\mu_{n-1}< \lambda_n\] be $2n-1$ real numbers such that \[\lambda_j=-\lambda_{n+1-j},\] for all  $j=1,\ldots,n$, and \[\mu_k=-\mu_{n-k},\] for all  $k=1,\ldots,n-1$. If $G$ has a spanning NEB tree $T$ at a vertex $v$, then there exists a skew-symmetric matrix $A$ in $S^-(T)$ with eigenvalues $\i \lambda_1,\i \lambda_2, \ldots, \i \lambda_n$ such that the eigenvalues of $A(v)$ are $\i \mu_1,\i \mu_2, \ldots, \i \mu_{n-1}$.
\end{theorem}

\begin{proof}
Without loss of generality assume that $v=n$.  By Theorem  \ref{main} there is an $A\in S^-(T)$ such that $A$ has eigenvalues $\i\lambda_1, \dots, \i\lambda_n$, $A(n)$ has eigenvalues  $\i\mu_1, \dots, \i\mu_{n-1}$, and $A$ has the Duarte-property with respect to $n$. By Theorem \ref{JAC}, the Jacobian of the $f$ defined in (\ref{eff}) evaluated at $A$ is nonsingular. Thus, the Jacobian  matrix of the function $g$ defined by (\ref{gee}) at $A$ is nonsingular.

The rest of the proof is similar to that of Theorem 4.2 of \cite{ms}. Assume that $G$ has $r$ edges not in $T$ and let $y_1, \ldots, y_r$ be $r$ new variables other than $x_1, \ldots, x_{n-1}$. We can extend the function $g: \mathbb{R}^{n-1} \rightarrow \mathbb{R}^{n-1}$ to a function $\hat{g}: \mathbb{R}^{(n-1)+r }\rightarrow \mathbb{R}^{n-1}$  by replacing each pair of entries of $M$ (and $N$) corresponding to an edge of $G$ not in $T$ by one of the $y_i$'s. Let $\hat{g}(x,y)$ be the vector of nonleading coefficients of the characteristic polynomials of $M$ and $N$. Let $\hat{g}(x,y)\atA = (c,d)$.

Since each of the $n-1$ entries of $A$ corresponding to the variable $x_j$ is nonzero, there is an open neighborhood $U$ of $(a_{i_1,j_1},\ldots,a_{i_{n-1},j_{n-1}}, 0, \ldots, 0)$ each of whose elements has no zeros in its first $n-1$ entries. By Theorem \ref{IFT},  there is an open neighborhood $V$ of $(a_{i_1,j_1},\ldots,a_{i_{n-1},j_{n-1}})$ and an open neighborhood $W$ of $(0,0,\ldots,0)$ such that $V\times W \subseteq U$ and for each $y \in W$ there is an $x\in V$ such that $\hat{g}(x,y)=(c,d)$. Take $y$ to be a vector in $W$ with no zero entries. Then, the graph of the matrix obtained from this choice of $x$ and $y$, $\widehat{A}$, is $G$, and also $\hat{g}(x,y)=(c,d)$, that is, the $\i\lambda_j$'s are the eigenvalues of $\widehat{A}$ and the $\i\mu_j$'s are the eigenvalues of $\widehat{A}(n)$. 
\end{proof}

Given $\lambda_1< \lambda_2 < \cdots < \lambda_n$ it is easy to find 
$\mu_1, \mu_2, \ldots, \mu_{n-1}$ such that (\ref{cauchysymm}) holds. Hence Theorem \ref{mainconnected} and Observation \ref{matching} immediately imply the following corollary.

\begin{corollary}
Let $G$ be  a connected graph on $n$ vertices and $\lambda_1$, $\lambda_2$, \ldots, $\lambda_n$ 
distinct real numbers such that \[\lambda_j=-\lambda_{n+1-j},\] for all  $j=1,\ldots,n$. If $G$ has a spanning tree which is NEB at a vertex, then $\match(G)=\lfloor \frac{n}{2}\rfloor$ and there exists a matrix $A\in S^-(G)$ with eigenvalues $\i\lambda_1, \ldots, \i\lambda_n$.
\end{corollary}

\begin{example}\label{lastexample}
We want to find a matrix whose graph is $C_4$, the cycle of length $4$, with eigenvalues $\pm\i, \pm 2 \i$, and its eigenvalues after deleting the forth row and columns are $0, \pm 1.5 \i$. Consider the matrix $A=[a_{ij}]$ constructed in Example \ref{exampleP4} (mentioned below). 
\[A \simeq \left[\begin{array}{cccc}
0 & 1.206045 & 0 & 0 \\
-1.206045 & 0 & 0.8918826 & 0 \\
0 & -0.8918826 & 0 & 1.658312 \\
0 & 0 & -1.658312 & 0
\end{array}\right].\]

Construct the matrices $M$ and $N$ as is Example \ref{exampleP4MN}:
\[M=\left[\begin{array}{cccc}
0&x_1&0&0\\
-x_1&0&x_2&0\\
0&-x_2&0&x_3\\
0&0&-x_3&0
\end{array}\right], \qquad N=M(4)=\left[\begin{array}{ccc}
0&x_1&0\\
-x_1&0&x_2\\
0&-x_2&0\\
\end{array}\right].\]

And set of the function $f$ as in Example \ref{exampleP4MNEff}:
\[f(x_1,x_2,x_3)= \left(-\frac{x_{1}^{2} +   x_{2}^{2} + x_{3}^{2}}{2}, \frac{(x_{1}^{4} + x_{2}^{4} + x_{3}^{4}) + 2 x_{2}^{2} ( x_{1}^{2} + x_{3}^{2} )}{4}, -\frac{x_1^2 + x_2^2}{2} \right).\]

Note that \[f\atA = \left(\frac{\tr M^2}{4}, \frac{\tr M^4}{8}, \frac{\tr N^2}{4}\right)\atA \simeq (-2.5, 4.25, -1.125),\] and

\[ \jac(f) = \left[\begin{array}{ccc}
 -x_1 &  -x_2 &  -x_3 \\
x_1^3 + x_1 x_2^2 & x_1^2 x_2 + x_2^3 + x_2 x_3^2 & x_2^2 x_3 + x_3^3 \\
 -x_1 & -x_2 & 0
\end{array}\right]. \]

Also note that $\det(\jac(f)\atA) = - x_1 x_2 x_3^3 \atA \simeq 4.9053 \neq 0$. Hence by the Implicit Function Theorem, for small perturbations of $a_{14}$ from $0$ to $\varepsilon$, there are $\widehat{a}_{12},\widehat{a}_{2,3},\widehat{a}_{3,4}$ such that \[\left(\frac{\tr \widehat{A}^2}{4}, \frac{\tr \widehat{A}^4}{8}, \frac{\tr \widehat{A}(1)^2}{4}\right) \atA \simeq (-2.5, 4.25, -1.125).\]
For example if $\varepsilon = 0.1$, then $\widehat{a}_{12} \simeq 1.257633,\widehat{a}_{2,3} \simeq 0.8175322,\widehat{a}_{3,4} \simeq 1.655294$, and 

\begin{minipage}[b]{0.45\textwidth}
\[ \widehat{A} \simeq \left[\begin{array}{cccc}
0 & 1.257633 & 0 & 0.1 \\
-1.257633 & 0 & 0.8175322 & 0 \\
0 & -0.8175322 & 0 & 1.655294 \\
-0.1 & 0 & -1.655294 & 0
\end{array}\right] \]
\end{minipage} \begin{minipage}[c]{0.45\textwidth}
\begin{flushright}
\begin{tikzpicture}[scale=.8,colorstyle/.style={circle, draw=black!100,fill=black!100, thick, inner sep=0pt, minimum size=1.2 mm}]
\node (1) at (1,1)[colorstyle,label=left:$1$]{};
\node (2) at (1,0)[colorstyle,label=left:$2$]{};
\node (3) at (2,0)[colorstyle,label=right:$3$]{};
\node (4) at (2,1)[colorstyle,label=right:$4$]{};
\draw [thick] (1)--(2)--(3)--(4)--(1);
\node () at (1.5,-1) {$C_4$};
\end{tikzpicture}
\end{flushright}
\end{minipage}\\

It is easy to verify that the eigenvalues of $\widehat{A}$ are approximately $\pm\i, \pm 2 \i$, and the eigenvalues of $\widehat{A}(4)$ are approximately $0, \pm 1.5 \i$. Furthermore the graph of $\widehat{A}$ is a cycle of length $4$.
\end{example}

\section*{Acknowledgment}
We would like to thank our academic advisor Bryan Shader for inspirations, fruitful discussions, and coining the term ``Nearly Even Branching (NEB)''. We also would like to thank the anonymous referee for the quick review and valuable suggestions.

\end{document}